\documentclass{article}

\usepackage{arydshln}

\usepackage[utf8]{inputenc}
\usepackage{hyperref}
\usepackage{amsfonts,amssymb,amsthm,mathtools,amsmath}
\usepackage[english]{babel}

\usepackage{xcolor}

\usepackage[ruled,vlined]{algorithm2e}

\definecolor{mycolor}{RGB}{255,105,180}

\def\N{\mathbb{N}}
\def\bN{\mathbb N}

\def\cB{\mathcal B}
\def\cS{\mathcal S}
\def\cX{\mathcal X}

\newtheorem{definition}{Definition}[section]
\newtheorem{remark}[definition]{Remark}
\newtheorem{lemma}[definition]{Lemma}
\newtheorem{example}[definition]{Example}

\newtheorem{theorem}[definition]{Theorem}
\newtheorem{corollary}[definition]{Corollary}

\title{On reducible non-Weierstrass semigroups}

\author{
J. I. Garc\'{\i}a-Garc\'{\i}a
 \footnote{
     Departamento de Matem\'aticas/INDESS (Instituto Universitario para el Desarrollo Social Sostenible),
     Universidad de C\'adiz, E-11510 Puerto Real  (C\'{a}diz, Spain).
     E-mail: ignacio.garcia@uca.es.}\\
D. Mar\'{\i}n-Arag\'on
 \footnote{
     Departamento de Matem\'aticas,
     Universidad de C\'adiz, E-11510 Puerto Real  (C\'{a}diz, Spain).
     E-mail: daniel.marin@uca.es.}
     \\
F. Torres
 \footnote{
     IMECC/UNICAMP, R. Sérgio Buarque de Holanda, 651, Cidade Universitária “Zeferino Vaz”, 13083-859, (Campinas, SP, Brazil).
     E-mail: ftorres@unicamp.br.}
    \\
 A. Vigneron-Tenorio
 \footnote{
 	Departamento de Matem\'aticas/INDESS (Instituto Universitario para el Desarrollo Social Sostenible), Universidad de C\'adiz,
 	E-11406 Jerez de la Frontera (C\'{a}diz, Spain).
     E-mail: alberto.vigneron@uca.es.}
}

\begin{document}
\maketitle
\begin{abstract}
Weierstrass semigroups are well-known along the literature. We present a new family of non-Weierstrass semigroups which can be written as an intersection of Weierstrass semigroups. In addition, we provide methods for calculating non-Weierstrass semigroups with genus as large as desired.
\end{abstract}


{\small \emph{MSC-class:} {14H55, 11P70, 20M14}}

\section*{Introduction}

A {\em numerical semigroup} $H$ is an additive submonoid of the non-negative integers
$\bN$ whose complement $G(H)=\bN \setminus H$, the {\em set of gaps} of $H$, is finite;
its cardinality $g(H)=\# G(H)$ is the {\em genus} of $H$. The elements of $H$ will be
refereed as the {\em non-gaps} of $H$. A suitable reference for the background on numerical
semigroups that we assume is the book \cite{RGS}.

In Weierstrass Point Theory one associates a numerical semigroup $H(P)$ to any point $P$ of
a complex (projective, irreducible, non-singular, algebraic) curve $\cX$ in such a way that
its genus coincides with the genus $g$ of the underlying curve; see e.g. \cite[III.5.3]{FK}.
This semigroup is called the {\em Weierstrass semigroup} at $P$, and it is the set of pole
orders at $P$ of regular functions on $\cX\setminus\{P\}$. We have that the set of gaps of
$H(P)$ equals $\{1,\ldots,g\}$ for all but finitely many points $P$ which are the so-called
{\em Weierstrass points} of the curve; they carry a lot of information about the curve; see e.g. \cite[III.5.11]{FK}, \cite{Centina}.

A numerical semigroup $H$ is called {\em Weierstrass} if there is a pointed curve $(\cX,P)$  such that $H=H(P)$. Around 1893, Hurwitz \cite{Hurwitz} asked about the characterization of Weierstrass semigroups; long after that, in 1980, Buchweitz \cite{B1}, \cite{B2}  pointed out the following combinatorial symple criterion. Let $\cS(g)$ denote the collection of numerical semigroups of genus $g$, and for $n\geq 2$ an integer let
  \begin{equation}\label{eq1.1}
\cB_n(g)=\{H\in \cS(g):\, \# G_n(H)>(2n-1)(g-1)\},
  \end{equation}
be the subcollection of {\em $n$-Buchweitz semigroups} of genus $g$,
being $G_n(H)$ the $n$-fold sum of elements of the set of gaps of $H$. Then by the Riemann-Roch
theorem each element
of $\cB_n(g)$ is non-Weierstrass (see \cite[p. 122]{EH} for further historical information).
Although Buchweitz also pointed out that $\cB_n(g)\neq \emptyset$
for each $n$ and large $g$, Kaplan and Ye \cite[Thm. 5]{KY} noticed that
   $$
{\rm lim}_{g\to \infty}\frac{\# \cB(g)}{\# \cS(g)}=0\, ,
   $$
where $\cB(g)=\cup_{n\geq 2}\cB_n(g)$; in particular,
${\rm lim}_{g\to \infty}\# \cB_2(g)/\# \cS(g)=0$ which is a result
suggested by previous numerical computations in Komeda's paper \cite[\S2]{KomedaSForum}.
This shows in particular that it is difficult to give explicit examples of semigroups in
$\cB_2(g)$ for large $g$.

With the idea of studying Weierstrass semigroups and having methods to find $2$-Buchweitz semigroups, we introduce the concept of PF-semigroup.
A numerical semigroup is a PF-semigroup if its multiplicity is the difference between its genus and its type plus one, and its set of pseudo-Frobenius numbers coincides with its set of gaps greater than its multiplicity.
These semigroups are the key to prove that the set of Weierstrass semigroups is not closed by intersections.

Inspired by the definition of the Schubert index (see \cite{KomedaSForum}), we use sequences of positive integer numbers to construct families of PF-semigroups.
These sequences are used to calculate subfamilies of $2$-Buchweitz PF-semigroups, in particular, to generate semigroups of this type with large genera.
The main result of this work (Theorem \ref{pegando_secuencias}) gives us an operation such that from two sequences we obtain new sequences describing again families of $2$-Buchweitz PF-semigroups.

The content of this work is organized as follows. 
In the first section, we introduce the concepts of PF-semigroup, $m$-semigroup and $n$-Buchweitz semigroup, and study some of their properties.
Section 2 is devoted to the decomposition of PF-semigroups as the intersection of Weierstrass semigroups. 
Finally, in Section 3, we represent the PF-semigroups using the differences of their pseudo-Frobenius numbers and study conditions on those sets to obtain $2$-Buchweitz PF-semigroups. Theorem \ref{pegando_secuencias} will give us a method for \emph{pasting} these sequences to get $2$-Buchweitz PF-semigroups with genus as large as desired.

\section{Preliminaries and results}

Given a numerical semigroup $H$, the minimal set, according the inclusion, $\{h_1<\cdots<h_p\}$ such that $H=\{\sum_{i=1}^p x_ih_i\mid x_i\in\N\}$ is called the minimal system of generators of $H$, and we write $H=\langle h_1,\ldots,h_p\rangle$. The element $h_1$ is called the multiplicity of $H$ and it is denoted by $m(H)$.

The pseudo-Frobenius numbers of $H$
are the elements of $PF(H)=\{x\in G(H)\mid \forall h\in H\setminus\{0\},\, x+h\in H\}$. The cardinality of this set is the type of $H$, denoted by $t(H)$. Note that $\max G(H)=\max PF(H)$. This number is called the Frobenius number of $H$ and denoted by $Fb(H)$. It is trivial to check that if $x>Fb(H)$ then $x\in H$, and  the number $c(H)=Fb(H)+1$ is the conductor of $H$.

\begin{definition}
A numerical semigroup $H$ is a $PF$-semigroup if $G(H)$ is the set $\{1,\ldots ,g(H)-t(H)\}\sqcup PF(H)$, and $\min PF(H)>m(H)$.
\end{definition}

\begin{remark}
Note that if $H$ is a $PF$-semigroup then $m(H)=g(H)-t(H)+1$.
\end{remark}

\begin{lemma}\label{condition1}
Every $PF$-semigroup $H$ has Frobenius number odd, in particular, $Fb(H)= 2g(H)-2t(H) +1$.
\end{lemma}

\begin{proof}
Let $H$ be a $PF$-semigroup, $t=t(H)$ and $g=g(H)$, then $G(H)=\{1,\ldots,g-t\}\sqcup PF(H)$.

Since $m(H)=g-t+1$, we assume there exists $h\in G(H)$ such that $h>2g-2t+1$. 
The element $2m(H)=2g-2t+2$ belongs to $H$, so $h> 2g-2t+2$. In this case, $h=d(g-t+1)+r$ where the integers $d$ and $r$ satisfy that $d\ge 2$ and $r\in [1,g-t]$. Thus, $h=(d-1)(g-t+1)+g-t+1+r$. Note that $(d-1)(g-t+1)\in H$ and as $h\not\in H$, then $g-t+1+r\in G(H)\setminus PF(H)$. That is, $H$ is not a $PF$-semigroup. We can affirm every integer greater than or equal to $2g-2t+2$ belongs to $H$.

Now, since $g-t+h\ge 2g-2t+2$ for every $h\in H\setminus\{0,g-t+1\}$ and $g-t\in G(H)\setminus PF(H)$, $2g-2t+1=(g-t)+(g-t+1)$ is a gap of $H$. Therefore, $Fb(H)=2g-2t+1$.
\end{proof}

The condition $Fb(H)= 2g(H)-2t(H) +1$ from previous lemma does not imply that $H$ is a $PF$-semigroup. For example, the semigroup $\langle 5, 6, 14\rangle$ satisfies $Fb(H)= 2g(H)-2t(H) +1$ but it is not a $PF$-semigroup.

\begin{definition}(See \cite{DecompositionRosales})
We say that a numerical semigroup is irreducible if it cannot be expressed as an intersection of two numerical semigroups containing it properly.
\end{definition}

\begin{definition}
A $m$-semigroup is a numerical semigroup with multiplicity $m$.
\end{definition}

Throughout this paper a $g$-semigroup is numerical semigroup with multiplicity and genus $g$. Note that by \cite[Theorem 14.5]{Pinkham}, all $g$-semigroups are Weierstrass.

Given $G$ the set of gaps of a numerical semigroup $H$, we denote by $G_n(H)$ the set $\{g_1+\cdots +g_n\mid g_1,\ldots,g_n\in G\}$. Note that the set $G_n(H)$ is known as $nG=G+G+\dots+G$ in numerical additive theory.

\begin{definition}
A numerical semigroup $H$ is called $n$-Buchweitz semigroup for some $n\geq 2$ if the cardinality of $G_n(H)$ is strictly greater than $(2n-1)(g(H)-1)$.
\end{definition}

Several papers study $n$-Buchweitz semigroups due to it is known that these semigroups are non-Weierstrass (see \cite{B1} or \cite{KY}). For example, in \cite[page 161]{KomedaSForum}, it is introduced a computational result showing that there are not $2$-Buchweitz semigroups with genus strictly smaller than 16. 
\section{Irreducible decomposition of $PF$-semigroups}

The decomposition of a numerical semigroup as the intersection of irreducible numerical semigroups has been studied in several papers (see \cite{DecompositionRosales} and the references therein). In this section, a decomposition of $PF$-semigroups by means of irreducible Weierstrass semigroups is introduced. We use this decomposition to show that the set of Weierstrass semigroups is not closed by intersection.

\begin{lemma}\label{lemma_decomposition}
Let $H$ be a $PF$-semigroup with type $t$ and set of gaps $G(H)=\{1,\ldots ,g(H)-t\}\sqcup PF(H)$. Then there exist $H_1,\ldots ,H_{t}$ irreducible $g_i$-semigroups of genus $g_i$ respectively such that $H=H_1\cap \cdots \cap H_t$.
\end{lemma}

\begin{proof}
Assume that $PF(H)=\{f_1<\cdots <f_t=Fb(H)\}$. For $i\in \{1,\ldots,t\}$, we define the $g_i$-semigroup $H_i$ given by the set of gaps $\{1,\ldots ,g_i-1,2g_i-1\}$ with $g_i=(f_i+1)/2$ if $f_i$ is odd, or by $\{1,\ldots ,g_i-1,2g_i-2\}$ with $g_i=(f_i+2)/2$ in other case.

In order to prove that $H=H_1\cap \cdots \cap H_t$, we show that $G(H)=G(H_1)\cup \cdots \cup G(H_t)$. By construction, both sets of gaps are equal for every element greater than or equal to $f_1$.

For any semigroup $H_i$, $G(H_i)\setminus \{f_i\}=\{1,\ldots, g_i-1\}$. So, $G(H)$ is equal to $G(H_1)\cup \cdots \cup G(H_t)$ if and only if $\max_{i\in \{1,\ldots,t\}}\{g_i-1\}=g(H)-t$. Note that by Lemma \ref{condition1} $Fb(H)$ is odd, and that maximal element is $g_{t}-1=(Fb(H)+1)/2-1$. Since $Fb(H)=2g(H)-2t+1$, equality $\max_{i\in \{1,\ldots,t\}}\{g_i-1\}=g(H)-t$ holds.
\end{proof}

Since every $g$-semigroup is Weierstrass, we obtain the following result.

\begin{corollary}
Any $PF$-semigroup is the intersection of Weierstrass semigroups.
\end{corollary}

\begin{example}\label{examples_Oliveira}
Let  $H$ and $H'$ be the semigroups given by the set of gaps
$$G(H)=\{ 1,\ldots ,g-4,2g-13,2g-11,2g-8,2g-7\}$$
and
$$G(H')=\{ 1,\ldots ,g-3,2g-29,2g-9,2g-5\}.$$
We have that both families of semigroups are $2$-Buchweitz and $4$-Buchweitz semigroups for genus $g$ greater than or equal to $16$ and $99$, respectively (see \cite[Example 1]{OliveiraSForum}). These semigroups can be decomposed as Lemma \ref{lemma_decomposition}:
\begin{enumerate}
    \item $H=H_1\cap \cdots \cap H_4$ where
        \begin{enumerate}
            \item $G(H_1)= \{1,\ldots g-7\}\cup \{2g-13\}$,
            \item  $G(H_2)= \{1,\ldots ,g-6\}\cup \{2g-11\}$,
            \item $G(H_3)= \{1,\ldots ,g-4\}\cup \{2g-8\}$,
            \item and $G(H_4)= \{1,\ldots ,g-4\}\cup \{2g-7\}$.
        \end{enumerate}
    \item $H'=H'_1\cap H'_2 \cap H'_3$ where
        \begin{enumerate}
            \item $G(H'_1)= \{1,\ldots ,g-13\}\cup \{2g-29\}$,
            \item $G(H'_2)= \{1,\ldots ,g-5\}\cup \{2g-9\}$,
            \item and $G(H'_3)= \{1,\ldots ,g-3\}\cup \{2g-5\}$.
        \end{enumerate}
\end{enumerate}
\end{example}

Note that, in general, the intersection of $g_i$-semigroups with genus $g_i$ is not a $PF$-semigroup. For example, fixing a set of $g_i$-semigroups with genus $g_i$ where the maximum of its Frobenius numbers is even, its intersection is not a $PF$-semigroup.

From Example \ref{examples_Oliveira}, we obtain that intersection of Weierstrass semigroups is not necessarily Weierstrass, so we can affirm that the set of Weierstrass semigroups is not closed by intersection.

\begin{corollary}
The set of Weierstrass semigroups is not closed by intersection.
\end{corollary}

It is possible that some intersections to be a Weierstrass semigroup. Consider the semigroups $H_1$, $H_2$ and $H_3$ given by the gap-sets $\{1,2,3,6\}$, $\{1,2,3,4,8\}$ and $\{1,2,3,4,9\}$, respectively. These semigroups are $4$-semigroup and $5$-semigroups, respectively. Their intersection is the $PF$-semigroup $H$ generated by $\{5,7,11,13\}$. The semigroups $H_1$, $H_2$ and $H_2$, as well as $H$ are Weierstrass because their multiplicity are smaller than or equal to 5 (see \cite[Page 42]{commutative-grillet}\label{t:weierstrassm5}). 

\section{Computing $2$-Buchweitz $PF$-semigroups}

In this section we focus on the study of $PF$-semigroups $H$ with genus $g$, type $t$ and gap-set $G(H)=\{1,\ldots ,g-t, 2(g-t+1)-a_1,\ldots , 2(g-t+1)-a_t \}$, where $a_1,\ldots ,a_t\in \N$ satisfying that $a_1>a_2>\cdots >a_{t-1}>a_t=1$.
To simplify, we denote 
$m(H)=g-t+1$ by $m$. Then, $G(H)=\{1,\ldots ,g-t, 2m-a_1,\ldots , 2m-a_t \}$.

\begin{theorem}\label{existencia_with_gaps}
Let $\{a_1>a_2>\cdots > a_{t-1}>a_t=1\}\subset \N$ and $t\ge 2$, for every  non-negative integer $g$ satisfying \begin{enumerate}
    \item $g\ge 2a_1+t-1$;
    \item\label{cond_card} $\#\big\{a_i+a_j\mid i,j\in\{1,2,\ldots,t\}\big\}> 3(t-1)$;
\end{enumerate}
the semigroup $H$ with gap-set  $G(H)=\{1,\ldots ,g-t, 2m-a_1,\ldots , 2m-a_t \}$ is a $2$-Buchweitz $PF$-semigroup.
\end{theorem}

\begin{proof}
Prove that H is a PF-semigroup. Note that for every $h\in H\cap [m,2m-1]$, $h+m>2m-1$ and then, $h+m\in H$. Thus, the elements $2m-a_1,\ldots$, and $2m-a_t$ are pseudo-Frobenius numbers of $H$. Consider now the subset $\{1,\ldots ,g-t\}\subset G(H)$. For all $k\in\{1,\ldots, m-a_1\}$, $2m-a_1-k\in H$ and $k+2m-a_1-k=2m-a_1\in G(H)$. That means $\{1,\ldots ,m-a_1\}\subset PF(H)$. Analogously, for every $k\in \{a_1,\ldots ,g-t\}$, using that $2m-1-k\in H$, we obtain $k+2m-1-k =2m-1\in G(H)$, and then $\{a_1,\ldots, g-t\}\subset PF(H)$. By hypothesis, $g$ is greater than or equal to $2a_1+t-1$, thus $\{1,\ldots , m-a_1\}\cap \{a_1,\ldots ,g-t\}$ is no empty and $m<2m-a_1$. Therefore, $H$ is a PF-semigroup.

We describe now explicitly the set $G_2(H)$ (recall that $m=g-t+1$),
$$
\begin{array}{ccl}
G_2(H) & = &  \{2,\ldots ,2m-2\} \cup \{2m-a_1+1,\ldots ,3m-a_1-1\}  \\
   &   & \cup \{2m-a_2+1,\ldots ,3m-a_2-1 \}\cup \cdots \\
   &   & \cup \{2m-a_{t-1}+1,\ldots ,3m-a_{t-1}-1 \}\cup \{2m,\ldots ,3m-2\} \\
   &   & \cup \{4m-2a_1,4m-a_1-a_2,\ldots ,4m-a_{t-1}-1,4m-2\}.
\end{array}
$$
Since $a_1>1$, we have $2m-a_1+1\le 2m-2$. Also, using that $g\ge 2a_1+t-1$, $m\ge 2a_1>a_1+1>a_2+1> \cdots >a_{t}+1$, and then $2m-a_{i+1}+1\le 3m-a_{i}-1\le 3m-a_{i+1}-1< 4m-2a_1$ for every $i\in \{1,\ldots ,t-1 \}$. Therefore, $G_2(H)=\{2,3,\ldots ,3m-1,3m-2\}\sqcup \{4m-2a_1,4m-a_1-a_2,\ldots ,4m-2\}$, and its cardinality is $3m-3+\#\big\{a_i+a_j\mid i,j\in\{1,2,\ldots,t\}\big\}$. By Condition \ref{cond_card}, this cardinality is strictly greater than $3(g-1)$. Thus, $H$ is a $2$-Buchweitz semigroup.

\end{proof}

In \cite{KomedaSForum}, the concept of Schubert index is defined: given a numerical semigroup with gap-set $\{l_1<l_2<\cdots < l_g\}$, for any $i=0,\ldots ,g-1$, set $\alpha_i= l_{i+1}-i-1$, the tuple $(\alpha_0,\ldots ,\alpha_{g-1})$ is called the Schubert index associated with the semigroup.
Moreover, \cite[Proposition 2.2]{KomedaSForum} introduces some $2$-Buchweitz semigroups that are $PF$-semigroups. 
For example, the case (3) into this proposition studies the semigroup of genus $g=q+p+1$ and Schubert index $\alpha =(0,\ldots ,0,q-2p,q-2p,q-2p+1,\ldots ,q-2p+p-1,q-2p+p-1)\in \{0\}^q\times \N^{g-q}$, with $q\ge 4p$ and $p\ge 3$ ($q,p\in \mathbb{N}$). Its gap-set is $\{1,\ldots q, 2q-2p+1,2q-2p+2,2q-2p+2+2\cdot 1,\ldots ,2q-2p+2+2(p-1)=2q,2q+1\}$.
Since $q+1+(2q-2p+1)= 3q-2p+2\ge 2q+4p-2p+2> 2q+2$, $(q+1)+\{2q-2p+1,2q-2p+2,2q-2p+2+2\cdot 1,\ldots ,2q-2p+2+2(p-1)=2q,2q+1\}$ is contained in the semigroup, the elements in $\{2q-2p+1,2q-2p+2,2q-2p+2+2\cdot 1,\ldots ,2q-2p+2+2(p-1)=2q,2q+1\}$ are pseudo-Frobenius numbers.
If we consider any $j\in\{2p+1,\ldots , q\}$, $2q+1 = j+  (2q+1) - j $, and the elements in $\{2p+1,\ldots , q\}$ are not pseudo-Frobenius numbers.
Something similar happens to the numbers in $\{1,\ldots , 2p \}$. In this case, $2q-2p=j+(2q-2p)-j$ for all $j\in\{1,\ldots ,2p \}$. That is to say the semigroup in Proposition 2.2 (3) is a $PF$-semigroup.
Analogously, it can be proved that the types (1) and (2) are $PF$-semigroups.

Given a sequence $\mathbf{d}=(d_1, \ldots, d_{t-1})\in(\mathbb{N}\setminus\{0\})^{t-1}$, we consider the numerical semigroup with gap-set $\{1,\ldots ,g-t,2m-1-\sum_{i=1}^{t-1}d_{i} < \cdots< 2m-1-d_{t-2}-d_{t-1} < 2m-1-d_{t-1}< 2m-1\}$. For a PF-semigroup $H$ with $G(H)=\{1,\ldots ,g-t, 2m-a_1,\ldots , 2m-a_t \}$, the set $G(H)$ is equal to $\{1,\ldots ,g-t,2m-1-\sum_{i=1}^{t-1}d_{i}, \ldots, 2m-1-d_{t-1}, 2m-1\}$ where $d_i=a_i-a_{i+1}$ for $i=1, \ldots ,t-1$.

For instance, the sequence of the above mentioned example of \cite{KomedaSForum} type (3) is $\mathbf{d}=(1,2,\ldots ,2,1)\in\mathbb{N}^{p+1}$, and the sequences for the types (1) and (2) are $(2,\ldots ,2,3,1)$ and $(1,3,2,\ldots ,2)$, respectively.
For a $PF$-semigroup with gap-set
$\{1,\ldots ,g-t,2m-1-\sum_{i=1}^{t-1}d_{i}, \ldots,  2m-1-d_{t-1}, 2m-1\},$
the relation between the sequence $\mathbf{d}=(d_1, \ldots, d_{t-1})\in(\mathbb{N}\setminus\{0\})^{t-1}$ and its Schubert index $\alpha=(\alpha_0,\ldots ,\alpha _{g-1})$ is the following: for any non-negative integer $i\leq g-t-1$, $\alpha _i=0$, and for any $i\in \{g-t+1,\ldots ,g-1\}$, $d_{i-g+t}=\alpha _{i} -\alpha_{i-1}+1$.

\begin{definition} Given a sequence $\mathbf{d}=(d_1, \ldots ,d_{t-1})\in(\mathbb{N}\setminus\{0\})^{t-1}$, we say that $\mathbf{d}$ is a $(g,t)$-Buchweitz sequence if there exists a $2$-Buchweitz $PF$-semigroup with genus $g$, type $t$ and gap-set
$G(H)=\{1,\ldots ,g-t,2m-1-\sum_{i=1}^{t-1}d_{i}, \ldots,  2m-1-d_{t-1}, 2m-1\}$.
\end{definition}

Several $(g,t)$-Buchweitz sequences can be constructed. The next example provides one.

\begin{example}
Consider the sequence ${\mathbf d}=(7,1,2,1)$ and $H$ the $PF$-semigroup with genus $g$ and $G(H)=\{1,\ldots, g-5\}\cup \{2g-20,2g-13,2g-12,2g-10,2g-9\}$. The set $G_2(H)$ is
$$
\begin{array}{ccl}
G_2(H) & = & \{2,\ldots ,2g-10\}\cup \{2g-19,\ldots ,3g-25\}\cup \{4g-40\} \\
    & & \cup \{2g-12,\ldots ,3g-18\} \cup \{4g-33,4g-26\} \\
    & & \cup\{2g-11,\ldots ,3g-17\} \cup \{4g-32,4g-25,4g-24\}  \\
    & & \cup\{2g-9,\ldots ,3g-15\} \cup \{4g-30,4g-23,4g-22,4g-20\}  \\
    & & \cup\{2g-8,\ldots ,3g-14\} \cup \{4g-29,4g-22,4g-21,4g-19, 4g-18\}\\
    & = & \{2,\ldots ,3g-14\}\cup \{4g-40,4g-33,4g-32,4g-30,4g-29\}\cup \\
    & & \bigcup_{i=0}^8 \{4g-26+i\}.\end{array}
$$
Note that if $3g-14 \leq 4g-40$, then $g\geq26$. If this occurs, the cardinality of the set $G_2(H)$ is greater than or equal to $3g-2$ and $3g-2>3(g-1)$. Thus, ${\mathbf d}$ is a $(g,5)$-Buchweitz sequence for every $g\geq 26$.
\end{example}

The following result determines the conditions that a given sequence must satisfy for being a $(g,t)$-Buchweitz sequence for each integer large enough $g$. Moreover, we prove that the reverse of a $(g,t)$-Buchweitz sequence is also a $(g,t)$-Buchweitz sequence.

\begin{corollary}\label{existencia}
Let $\mathbf{d}=(d_1,\ldots, d_{t-1})\in(\mathbb{N}\setminus\{0\})^{t-1}$ and $\mathbf{d'}=(d_{t-1},d_{t-2},\ldots, d_1)\in(\mathbb{N}\setminus\{0\})^{t-1}$. Then, for every non-negative integer $g$ satisfying
\begin{enumerate}
    \item $g\geq 2\sum_{i=1}^{t-1}d_i+t+1$;
    \item $\#\big\{\sum_{i=n}^{t-1}d_i+\sum_{j=m}^{t-1}d_j\mid n,m\in\{1,2,\ldots,t-1\}\big\} > 3(t-1)$;
\end{enumerate}
the sequences ${\mathbf d}$ and ${\mathbf d'}$  are  $(g,t)$-Buchweitz sequences.
\end{corollary}
\begin{proof}
Proceed similarly to the proof of Theorem \ref{existencia_with_gaps} taking $d_i=a_i-a_{i+1}$ for $i=1, \ldots ,t-1$, and $d'_i=d_{t-i}$  for $i=1, \ldots ,t-1$.
\end{proof}

Note that from the previous result we obtain an algorithm for checking the existence of $2$-Buchweitz $PF$-semigroups for a fixed sequence. The sketch of this algorithm is Algorithm \ref{algoritmo_check_PFSemig}.

\begin{algorithm}[H]\label{algoritmo_check_PFSemig}
	\KwIn{$\mathbf{d}=(d_1,\ldots ,d_{t-1})\in(\mathbb{N}\setminus\{0\})^{t-1}$}
	\KwOut{ The set of $g\in\N$ satisfying 1 and 2 in Corollary \ref{existencia}.
	}
\Begin
{
    $f\leftarrow 2g-2t+1$\\
    $pf\leftarrow \{f\}$\\
    \For{$i\leftarrow t-1$ \KwTo $1$}
    {
        $pf\leftarrow\{pf-d_i\}\cup pf$\;
    }
    $cardAB\leftarrow (f+g-t)+1$\;
    $cardC\leftarrow\#(2pf)$\;
    $G_2\leftarrow cardAB + cardC$\;
    $ineq\leftarrow \{G_2 >3(g-1)\}\cup\{g\geq 2\sum_{i=1}^{t-1}d_i+t+1\}$\;
    $g\leftarrow solve(ineq)$\;
	\Return $g$
}
\caption{Algorithm to check if a sequence is the sequence of a family of $2$-Buchweitz PF-semigroups.}
\end{algorithm}

We illustrate Corollary \ref{existencia} and Algorithm \ref{algoritmo_check_PFSemig} with some easy examples.

\begin{example}\label{exK}
Fix the sequence $\mathbf{d}=(1,3,3,2)$, the minimal integer $g$ obtained from Algorithm \ref{algoritmo_check_PFSemig} is $g=24$. Thus, the semigroups with genus greater than or equal to $24$ and associated sequence $\mathbf{d}$ are $2$-Buchweitz $PF$-semigroups. 

To obtain this type of sequence is really simple and other examples are given by the following sequences: $(1,4,3)$ for $g\geq 22$, $(2,4,3)$ for $g\geq 23$, etc.
\end{example}

Table \ref{Table} compares the number of numerical semigroups, $2$-Buchweitz semigroups and $2$-Buchweitz $PF$-semigroups. Some of the results appearing in this table are included in \cite{KomedaSForum}.
\begin{table}[h]
\centering
{
\begin{tabular}{|c|c|c|c|}
  \hline
  Genus & NS & 2-BS & 2-BPFS  \\
  \hline
16  & 4806 & 2 & 2   \\ \hdashline
17  & 8045 & 6 & 3    \\ \hdashline
18  & 13476 & 15 & 10    \\ \hdashline
19  & 22464 & 31 & 19    \\ \hdashline
20  & 37396 & 67 & 35    \\ \hdashline
21  & 62194 & 145 & 72    \\ \hdashline
22  & 103246 & 293 & 146    \\ \hdashline
23  & 170963 & 542 & 257    \\ \hdashline
24  & 282828 & 1053 & 469    \\ \hdashline
25  & 467224 & 1944 & 795    \\ \hdashline
26  & 770832 & 3591 & 1497    \\ \hdashline
27  & 1270267 & 6584 & 2655    \\ \hdashline
28  & 2091030 & 11871 & 4555    \\ \hdashline
29  & 3437839 & 20987 & 7745    \\ \hdashline
30  & 5646773 & 37598 & 13450    \\ \hdashline
31  & 9266788 & 66330 & 23108    \\ \hdashline
32  & 15195070 & 116501 & 38944    \\ \hdashline
33  & 24896206 & 203300 & 64873    \\ \hdashline
34  & 40761087 & 353978 & 110576    \\ \hdashline
35  & 66687201 & 615762 & 187966    \\ 
\hline
\end{tabular}
}
\caption{Number of numerical semigroups (NS) and $2$-Buchweitz semigroups (2-BS) compared with number of $2$-Buchweitz $PF$-semigroups (2-BPFS) up to genus $35$.}\label{Table}
\end{table}

Given two Buchweitz sequences, it can be constructed a new Buchweitz sequence. This result allows us to obtain $2$-Buchweitz semigroups with genus as large as wanted.

\begin{theorem}\label{pegando_secuencias}
Let $\mathbf{d}=(d_1,\ldots, d_{t-1})\in(\mathbb{N}\setminus\{0\})^{t-1}$ and $\mathbf{d'}=(d'_1,\ldots, d'_{h-1})\in(\mathbb{N}\setminus\{0\})^{h-1}$ be two $(g,t)$-Buchweitz and $(g',h)$-Buchweitz sequences (respectively) satisfying Corollary \ref{existencia}. For every integer $k\ge 1$, if  $d'_{h-1}> k$, then
$${\mathbf d''}=(d'_1,\ldots, d'_{h-1}, k ,d_1,\ldots, d_{t-1})$$
is a $(g'',t+h)$-Buchweitz sequence for every integer $g''\ge g+g'+2k-1$.
\end{theorem}
\begin{proof}
Let $A=\{a_1>\cdots > a_{t-1}>a_t=1\}\subset \N$ and $A'=\{a'_1>\cdots > a'_{h-1}>a'_h=1\}\subset \N$ be the sets such that $d_i=a_i-a_{i+1}$ for $i=1, \ldots ,t-1$, and $d'_i=a'_i-a'_{i+1}$ for $i=1, \ldots ,h-1$. Consider $a''_i= a_{i-h}$ for $i=h+1,\ldots ,t+h$, $a''_i= a'_i+a_1+k-1$ for $i=1,\ldots ,h$. Define $A''= \{a''_1, \ldots , a''_{h+t}=1\}$. We have that
$$A''=\left\{
a'_1+a_1+k-1, \ldots , a'_{h-1}+a_1+k-1,k+a_1, 
a_1,\ldots , a_{t-1},a_t=1\right\}.$$
Note that if $\mathbf{d}''=(d''_1,\ldots, d''_{h+t-1})$, then $d''_i=a''_i-a''_{i+1}$ for every $i=1, \ldots ,h+t-1$. Since $g''\ge g+g'+2k-1$, we obtain that $g''\ge 2\sum_{i=1}^{h-1}d'_i+h+1 + 2\sum_{i=1}^{t-1}d_i+t+1 +2k-1= 2a''_1+h+t-1$.

To determine if $\mathbf{d}''$ is a Buchweitz sequence, we study the cardinality of $2A''$. Note that $2A=\{2=2a_t< 1+a_{t-1} < \cdots <2a_1\}$, $2A'=\{2=2a'_h< 1+a'_{h-1} < 1+a'_{h-2}<\cdots <2a'_1\}$, and
\begin{multline*}
2A''=\big\{2=2a_{t}< 1+a_{t-1} <\cdots <2a_1< \cdots < 2a_1+k < \cdots\\  <2a_1+2(k-1)+2<2a_1+2(k-1)+1+a'_{h-1}<\cdots  < 2a_1+2(k-1)+2a'_1\big\}.
\end{multline*}
Thus, 
$B=2A\sqcup (2(a_1+k-1)+2A')\subset 2A''$ and $\#(2A'')\ge 3(t-1)+1+3(h-1)+1$. 
Note that $2a_1+k\in 2A''\setminus B$. 

Let $a'_{h-1}+k-1+2a_1\in 2A''$, we know that  $a'_{h-1}+k-1+2a_1<2a_1+2(k-1)+1+a'_{h-1}$. Since $d'_{h-1}>k$, then $a'_{h-1}+k-1+2a_1>2a_1+2k$ and $a'_{h-1}+k-1+2a_1\in 2A''\setminus B$. 
Hence, $2A\sqcup (2(a_1+k-1)+2A')\sqcup \{2a_1+k,a'_{h-1}+k-1+2a_1\}\subset 2A''$ and therefore $\#(2A'')\ge 3(h-1)+3(t-1)+4= 3(h+t-1)+1>3(h+t-1)$. By Theorem \ref{existencia_with_gaps}, the semigroup associated with $\mathbf{d}''$ is a $2$-Buchweitz $PF$-semigroup for every integer $g''\ge g+g'+2k-1$.
\end{proof}

The condition $d'_{h-1}>  k$ in Theorem \ref{pegando_secuencias} cannot be removed. For $k=1$ and every large enough genera $g$ and $g'$, consider the (g,5)-Buchweitz sequence (1,2,2,1), and the (g',5)-Buchweitz sequence (2,3,1,1). For every genus, the numerical semigroups associated with the sequences (1,2,2,1,{\bf 1},1,2,2,1), and (2,3,1,1,{\bf 1},1,2,2,1) are not 2-Buchweitz.

From the examples of Buchweitz sequences obtained from Algorithm \ref{algoritmo_check_PFSemig} and their associated 2-Buchweitz semigroups, and by using Theorem \ref{pegando_secuencias}, it is easy to generate several $2$-Buchweitz semigroups with large genera.

\begin{example}
If we take the sequences $(2,4,3)$ and $(1,4,3)$ from Example \ref{exK}, we know that the following sequences are $(g,t)$-Buchweitz 

\begin{center}
\begin{tabular}{rl}
$(2,4,3)$ & for $g\geq 23$,\\ 
$(1,4,3, {\mathbf 2}, 2,4,3)$ &  for $g\geq 48$,\\ 
$(1,4,3, {\mathbf 2}, 1,4,3,2,2,4,3)$ & for $g\geq 73$,\\
$(1,4,3, {\mathbf 2}, 1,4,3,2,1,4,3,2,2,4,3)$ & for $g\geq 98$,\\
$\vdots$
\end{tabular}
\end{center}
In the same way, we can use other sequences to get $2$-Buchweitz semigroups with genera as large as we wish.
\end{example}

{\bf Acknowledgement}. Part of this paper was written during a visit of the third-named author to the Universidad de C\'{a}diz (Spain), his visit was  partially supported by {\em Ayudas para Estancias Cortas de Investigadores} (EST2018‐0, Programa de Fomento e Impulso de la Investigación y la Transferencia en la Universidad de Cádiz).

The third-named author was partially supported by CNPq/Brazil (Grant 310623/2017-0).

The first, second and fourth-named authors were partially supported by Junta de Andalucía research group FQM-366 and by the project MTM2017-84890-P (MINECO/FEDER, UE).

\end{document}